\documentclass{amsart}

\usepackage{latexsym,amsfonts} 
\usepackage{amsmath,amssymb,graphics,setspace}
\usepackage{amsthm}
\usepackage{MnSymbol}
\usepackage{mathrsfs} 
\usepackage{fontenc}%

\usepackage{paralist}
\usepackage{color}
\usepackage[verbose]{wrapfig}
\usepackage[ruled,vlined,linesnumbered]{algorithm2e}
\usepackage{marvosym}
\usepackage{picins,graphicx}
\usepackage{pstricks,pst-text,pst-grad,pst-node,pst-3dplot,pstricks-add,pst-poly,pst-coil,pst-tree} 
\usepackage{multido,pst-node,pst-bspline}
\usepackage{multirow}
\usepackage{caption}
\usepackage{subfigure}
\usepackage[all]{xy}
\renewcommand{\thesubfigure}{\thefigure.\arabic{subfigure}}
\makeatletter
\renewcommand{\p@subfigure}{}
\renewcommand{\@thesubfigure}{\thesubfigure:\hskip\subfiglabelskip}
\makeatother

\usepackage{hyperref}
\hypersetup{linktocpage=true,colorlinks=true,linkcolor=blue,citecolor=blue,pdfstartview={XYZ 1000 1000 1}}

\newcommand{\cl}{\mbox{cl}}

\newcommand{\Nrv}{\mbox{Nrv}}
\newcommand{\near}{\delta} 

\newcommand{\hnear}[1]{\text{$\delta^{#1}$}}

\newcommand{\hsn}[1]{\mathop{\delta^{#1}}\limits^{\doublewedge}}

\newcommand{\hdnear}[1]{\delta_{\Phi}^{#1}}

\newtheorem{example}{Example}
\newtheorem{remark}{Remark}
\newtheorem{definition}{Definition}
\newtheorem{lemma}{Lemma}
\newtheorem{theorem}{Theorem}

\usepackage[normalem]{ulem} 

\begin{document}

\title[Hyperconnected CW complex]{ Hyperconnected Relator Spaces.\\   CW Complexes and Continuous Function Paths that are Hyperconnected}

\author[M.Z. Ahmad]{M.Z. Ahmad$^{\alpha}$}
\email{ahmadmz@myumanitoba.ca}
\address{\llap{$^{\alpha}$\,}
Computational Intelligence Laboratory,
University of Manitoba, WPG, MB, R3T 5V6, Canada}
\thanks{\llap{$^{\alpha}$\,}The research has been supported by University of Manitoba Graduate Fellowship and Gorden P. Osler Graduate Scholarship.}

\author[J.F. Peters]{J.F. Peters$^{\beta}$}
\email{James.Peters3@umanitoba.ca}
\address{\llap{$^{\beta}$\,}
Computational Intelligence Laboratory,
University of Manitoba, WPG, MB, R3T 5V6, Canada and
Department of Mathematics, Faculty of Arts and Sciences, Ad\.{i}yaman University, 02040 Ad\.{i}yaman, Turkey}
\thanks{\llap{$^{\beta}$\,}The research has been supported by the Natural Sciences \&
Engineering Research Council of Canada (NSERC) discovery grant 185986 
and Instituto Nazionale di Alta Matematica (INdAM) Francesco Severi, Gruppo Nazionale per le Strutture Algebriche, Geometriche e Loro Applicazioni grant 9 920160 000362, n.prot U 2016/000036.}


\date{}

\dedicatory{Dedicated to J.H.C. Whitehead and Som Naimpally}

\begin{abstract}
This article introduces proximal cell complexes in a hyperconnected space.   Hyperconnectedness encodes how collections of path-connected sub-complexes in a Alexandroff-Hopf-Whitehead CW space are near to or far from each other. 
Several main results are given, namely, a hyper-connectedness form of CW (Closure Finite Weak topology) complex, the existence of continuous functions that are paths in hyperconnected relator spaces and hyperconnected chains with overlapping interiors that are path graphs in a relator space. 
An application of these results is given in terms of the definition of cycles using the centroids of triangles.  
\end{abstract}

\maketitle

\section{Introduction}\label{sec:intro}
This paper revisits the notions of path and connectedness in cell complexes that have an Alexandroff-Hopf-Whitehead closure-finite, weak topology on them.   J.H.C. Whitehead introduced CW topology more than 80 years ago, in his paper published in 1939~[pp. 315-317]\cite{Whitehead1939homotopy} and elaborated in 1949~~\cite[\S 5, p. 223]{Whitehead1949BAMS-CWtopology}.  This discovery was derived from two conditions for a cell complex introduced by P. Alexandroff and H. Hopf in their topology of complexes~\cite[\S III, starting on page 124]{AlexandroffHopf1935Topologie}.    To gain control of the relationships in CW complexes, we consider cell complexes equipped with a collection of one or more proximities called a proximal relator~\cite{Peters2016relator}, an extension of a Sz\'{a}z relator~\cite{Szaz1987}, which is a non-void collection of connectedness proximity relations on a nonempty cell complex $K$.   A space equipped with a proximal relator is called a \emph{proximal relator space}.    A natural outcome of our revisiting connectedness in CW complexes is the introduction of what is known as hyperconnectedness\cite[Sec.~2]{ahmad2018maximal}, which encodes how collections of path-connected sub-complexes in a space are near to or far from each other, either spatially~\cite{Naimpally70withWarrack,Naimpally2013} or descriptively~\cite{DiConcilio2018MCSdescriptiveProximities}.   A main result in this paper is that hyperconnected chains with overlapping interiors in a relator space are path graphs (see Theorem~\ref{thm:pathgraph}).

\section{Preliminaries}\label{sec:prelim}

Let $K$ denote a planar cell complex containing three types of cells, namely, vertex (0-cell), edge (1-cell) and filled triangle (2-cell).  Each $K$ has an Alexandroff-Hopf-Whitehead~\cite{Whitehead1949BAMS-CWtopology, AlexandroffHopf1935Topologie} closure finite, weak topology defined on it.   


Let $A$ be a nonempty set of path-connected vertices in the cell complex $K$ on a bounded region of the Euclidean plane $X$, $p$ a vertex in $A$.  An \emph{\bf open ball} $B_r(p)$ with radius $r$ is defined by
\[
B_r(p) = \left\{q\in K: |p - q| < r\right\}.
\]
The \emph{closure} of $A$ (denoted by $\cl A$) is defined by
\[
\cl A = \left\{q\in X: B_r(q)\subset A\ \mbox{for some $r$}\right\}\ \mbox{(Closure of set $A$)}.
\]
A CW complex  $K$ is a Hausdorff space with a decomposition satisfying the following conditions:
\begin{compactenum}[$1^o$]
\item \textbf{Closure finiteness:} closure of each cell $\cl A,\, A= \sigma^n \in K$, intersects a finite number of other cells.
\item \textbf{Weak topology:} $A \subset K$ is closed, provided $A \cap \cl \sigma^n \neq \emptyset$ is closed for all $\sigma^n \in K$.
\end{compactenum}

A complex $K$ with $n$ cells is denoted by $\sigma^n$.  A union of $\sigma^j \in K, \, j \leq n$ is called a \emph{n-skeleton} $K^n$. A \emph{fibre bundle},$(E,B,\pi,F)$, is a structure that describes a relation,$\pi:E \rightarrow B$, between the \emph{total space} $E$ and the \emph{base space} $B$. Here $\pi$ is a continuous surjection and $F \subset E$ is the \emph{fibre}. 
\begin{figure}
	\centering
	\begin{subfigure}[Descriptive CW complex]
		{\includegraphics[scale=0.5,bb=125 560 220 690,clip]{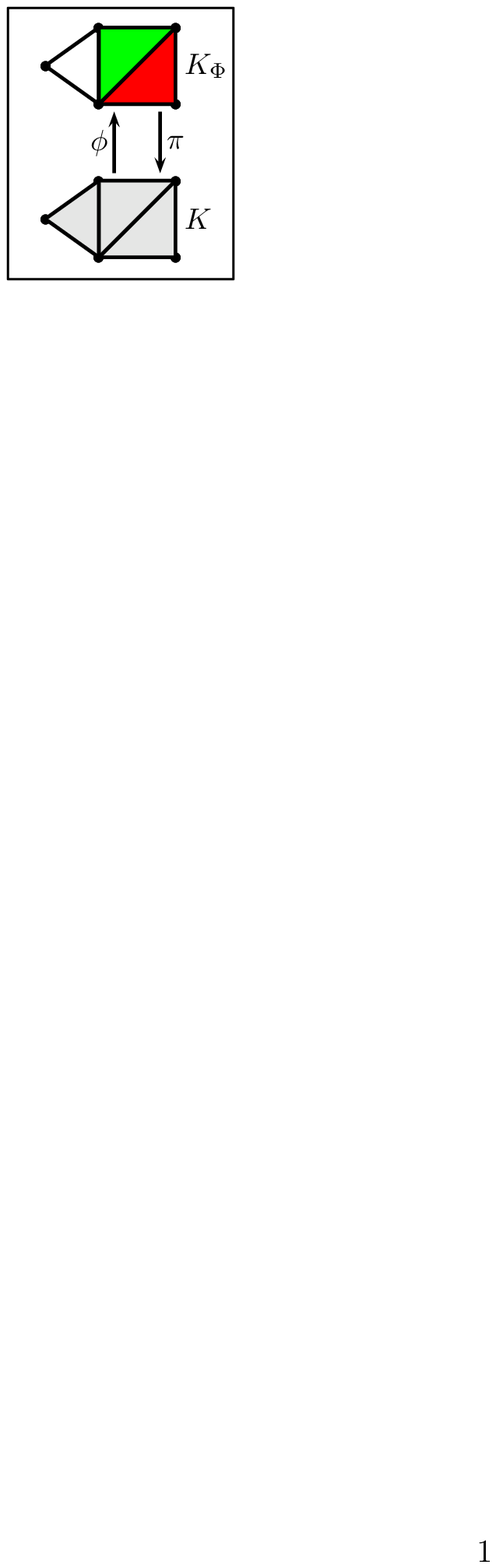}
			\label{subfig:desCW}}
	\end{subfigure}
	\begin{subfigure}[Local trivialization]
		{\includegraphics[scale=0.7,bb=120 610 350 670,clip]{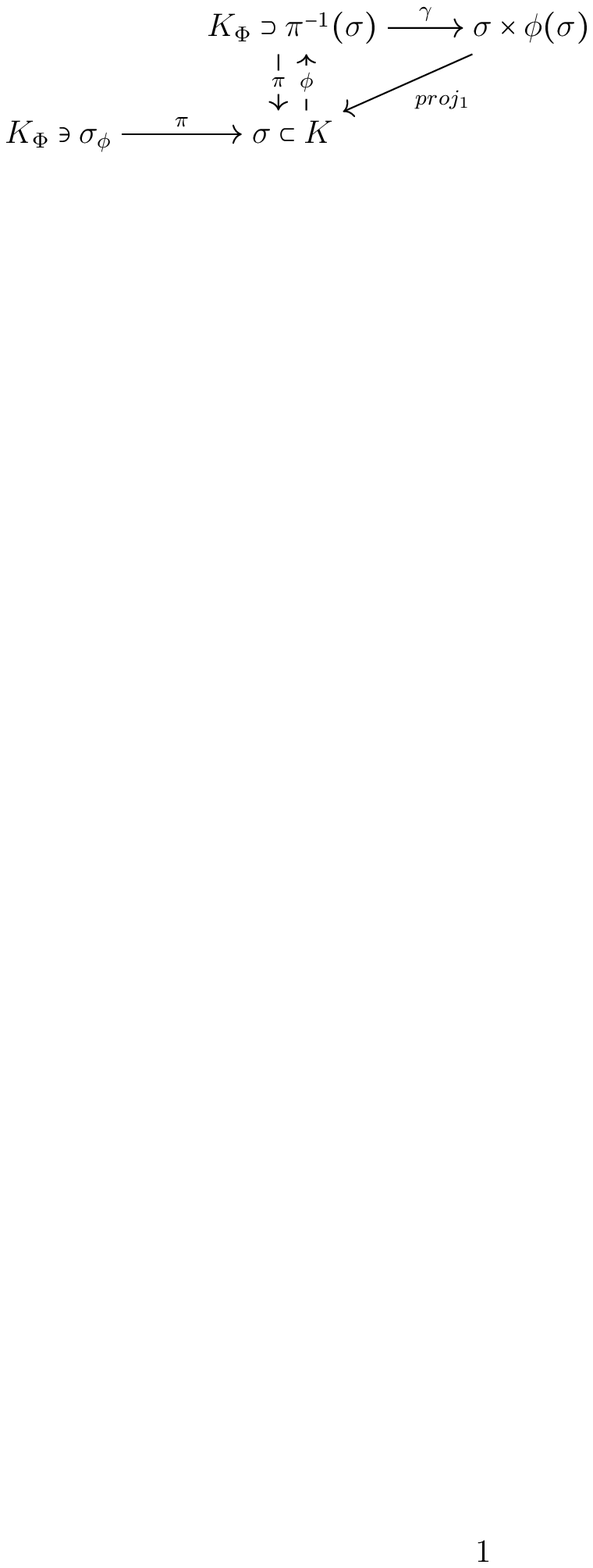}
		\label{subfig:triv}}
	\end{subfigure}
	\begin{subfigure}[Path-connectedness]
		{\includegraphics[scale=0.5,bb=125 560 220 690,clip]{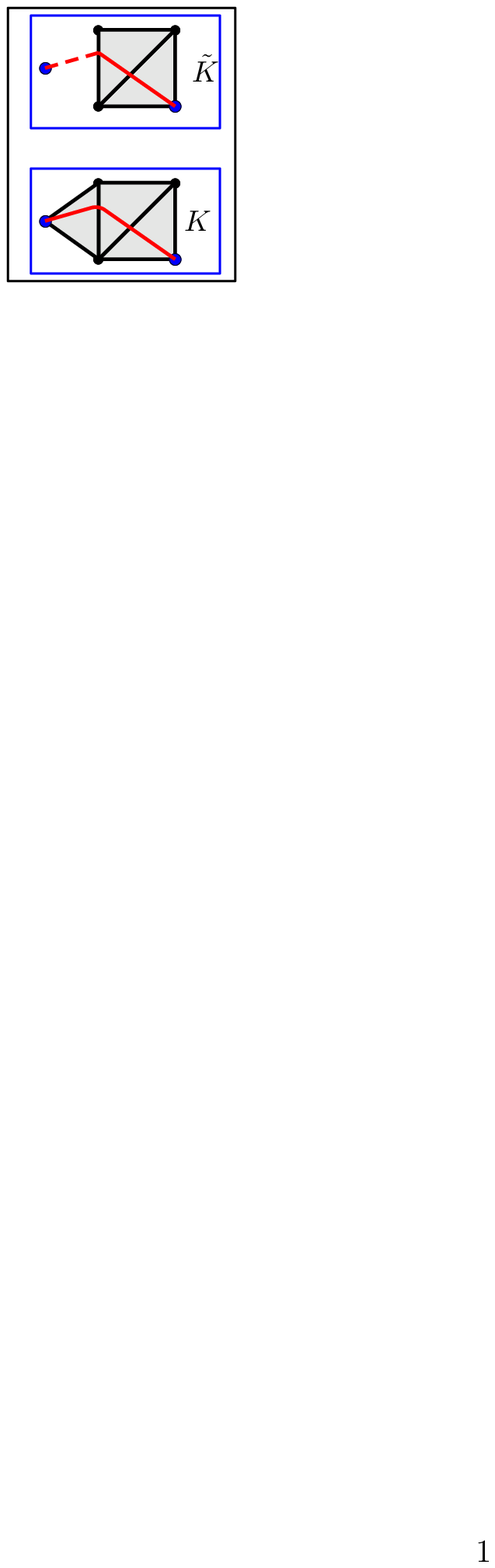}
		\label{subfig:pthcnctd}}
	\end{subfigure}
	\caption{Fig.~\ref{subfig:desCW} represents the descriptive CW complex as a fibre bundle $(K_\Phi,K,\pi,\phi(\sigma^n))$. The local trivialization property in this fibre bundle is illustrated as a commutative diagrams in Fig.~\ref{subfig:triv}. Fig.~\ref{subfig:pthcnctd} presents an example of a path-connected complex $K$, and a space $\tilde{K}$ which is not.}
	\label{fig:fig1}	
\end{figure}
A \emph{region based probe function} $\phi:2^K \rightarrow \mathbb{R}^n$ maps each set to its description. Then a \emph{descriptive cell complex} $K_\Phi$ can be defined as a fibre bundle, $(K_\Phi,K,\pi,\phi(\sigma^n))$, where $K$ is a cell complex and $\sigma^n \subset K$ is a simplex in $K$. A fibre bundle is a generalization of product topology and satisfies \emph{local trivialization condition},  stating that in a small neighborhood $\pi^{-1}(U) \subset E$ is homeomorphic to $U \times F$ via a map $\gamma$ as shown in Fig.~\ref{subfig:triv}. $U \subset B$, and the map $\pi^{-1}$ is called the \emph{section}. For $(K_\Phi,K,\pi,\phi(U))$ the region based probe $\phi$ is the \emph{section}, as it is homeomorphic to $\pi^{-1}$ for $U \subset B$. In topology a simplex with an empty interior is a \emph{hole}. We extend this notion to \emph{descriptive hole}, which is a region of constant description. Thus a descriptive hole is equivalent to the traditional definition of hole if the description is considered to be $\emptyset$.

\begin{example}
	Let us examine the notion of a descriptive CW as a fibre bundle, $(K_\Phi,K,\pi,\phi(\sigma^n))$. For this we consider Fig.~\ref{subfig:desCW}. Here $K$ is the base space and $\sigma^n \in K$ is a $n$-simplex. Let us consider a function $\phi:2^K \rightarrow \mathbb{R}^n$ which assigns to each $\sigma^n \in K$ a color. This process gives us the $K_\Phi$, in which there is a red, green and an empty(having no interior,$\phi(\sigma^2)=\emptyset$) triangle. Contrast this to $K$ in which every triangle has the same description. All the triangles in $K_\Phi$ is a descriptive hole as it has a constant description, while the triangle with $\phi(\triangle)=\emptyset$ is the hole in the traditional sense(having empty interior).
	
	We briefly explain the local trivialization of the fibre bundle, captured as a commutative diagram in Fig.~\ref{subfig:triv}. It shows that the $K_\Phi$ is a product space in a small neighborhood of $K$. The preimage of $\sigma \subset K$ under the map $\pi$ is homomorphic(via continuous map $\gamma$) to $\sigma \times \phi(\sigma)$, where $\phi$ is the section of this fibre bundle. Flexibility inherent in this structure allows us to account for the changing description(or color in this case) as we traverse the different $\sigma \in K$.  $\qquad \blacksquare$
\end{example}

\begin{figure}
	\centering
	\includegraphics[scale=0.7,bb=117 550 360 680,clip]{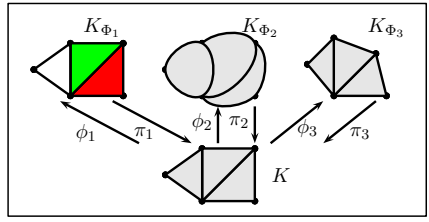}
	\caption{This figure illustrates how we can construct different descriptive CW complexes on the same base space i.e. $(K_{\Phi_i},K,\pi_i,\phi_i(U))$ for $i=1,2,3$. The function $\phi_1$ represents the hue, $\phi_2$ represents the curvature and $\phi_3$ represents area of the triangles.}
	\label{fig:multides}
\end{figure}

The descriptive CW complex can serve as a tool for abstraction to focus on a particular aspect of the underlying topological space. Using different probe functions we can construct distinct descriptive CW complexes atop the same base space. Let us consider an example to demonstrate this.
\begin{example}
	 Fig.~\ref{fig:multides} demonstrates the idea that we can construct different descriptive CW complexes on the same base space. Depending on the application the base space under analysis can have many features which are of interest. We can thus incorporate them on the same base space by the means of different fiber bundle structures by varying the projection $\pi$ and the corresponding section $\phi$. In this figure we have three different descriptive complexes $K_{\Phi_i}$ for $i=1,2,3$. The function $\phi_1$ attaches to each $\sigma^n \in K$ a value of hue or color. Similarly, $\phi_2$ attaches a value of curvature and $\phi_3$ the area. If we are looking at each individually we ignore the rest and focus on that particular property.
\end{example} 

Recall that hyperconnectedness encodes how collections of path-connected sub-complexes in a space are related(near or far) to each other. For a complete list of the axioms we refer the reader to \cite[Sec.~2]{ahmad2018maximal}. $\hnear{n}(A_1,\cdots,A_n)=0$ if the sets are \emph{near} and $\hnear{n}(A_1,\cdots,A_n)=1$ if \emph{far}, where $A_i \subset X$ for $i \in \mathbb{Z}^+$. A family of sets $A_i$ is Lodato hyperconnected, $\hnear{n}(A_1,\cdots,A_n)=0$, if they share a non-empty intersection \cite[axiom \textbf{(hP3)}]{ahmad2018maximal}. Strong hyper-connectedness, $\hsn{n}(A_1,\cdots,A_n)=0$, implies that interiors of the sets share non-empty intersection \cite[axiom \textbf{(snhN5)}]{ahmad2018maximal}. Descriptive hyper-connectedness $\hdnear{n}(A_1,\cdots,A_n)=0$, requires the sets to have a non-empty descriptive intersection i.e. they share a common description \cite[axiom \textbf{(dhP3)}]{ahmad2018maximal}.

\begin{example}
	We consider Fig.~\ref{subfig:pthcnctd}, to illustrate the concept of path-connectedness. If in a space there exists a path, in the space, between all the constituent points it is called path-connected. It is obvious that $K$ is a convex space, hence by definition path-connected. Whereas, in $\tilde{K}$ there exists no path within this space between the blue point and any of the other points. $\qquad \blacksquare$
\end{example}
 
Hyper-connectedness yields structures in a CW complex. \emph{Nerve}, $\Nrv$, is a collection of hyper-connected sets, with the number of sets befing the \emph{order}, $|\Nrv|$. Nerve with the maximal order is the \emph{maximal nuclear cluster}(MNC), with the common intersection as the \emph{nucleus}, $d$. Each of the sets in a MNC are the \emph{spokes}. A \emph{spoke complex} of order $n$, $skcx_n$, is a generalization of MNC. $skcx_k(d)$ is defined recursivly as the collection of sets hyper-connected to $skcx_{k-1}$, but far from $skcx_{k-2}$, with $skcx_0=d$. A cycle with the centroids of $skcx_k(d)$ as vertices is the \emph{$k^{th}$ maximal centroidal cycle}, $mcyc_k(d)$. The collection of $mcyc_i$ for $i \leq k$, is the \emph{$k^{th}$ maximal centroidal vortex}, $mvort_k(d)$. Each of these structures are related to their descritive counterparts via a fibre bundle in the similar fashion to the parent CW complex $K$.

\section{Main Results}\label{sec:mainres}
This section is divided into two subsections.
\subsection{Hyperconnected CW complex}\label{subsec:hypcw}
~\\
We can define a CW complex in terms of hyperconnectedness, beginnig with
\begin{lemma}
	Let $X$ be a space and $A\subset X$, then
	\begin{align*}
	clA=\{q\in X:\hnear{2}(A,q)=0\}
	\end{align*}
\end{lemma}
\begin{proof}
	We know from comparing \cite[axiom \textbf{(hP3)}]{ahmad2018maximal} and \cite[axiom \textbf{(P3)}]{ahmad2018maximal}, that $\hnear{2}(A,q)=0 \Leftrightarrow  A \near q$. Thus the equation presented in theorem is equivalent to $clA=\{q\in X:A \near q\}$ which has been established in \cite[\S 2.5, p.~439]{Cech1966}. 
\end{proof}
Now, we define the Hausdorff($T_2$) property of a space in terms of Lodato hyperconnectedness in a proximal relator space.
\begin{lemma}\label{lem:hausdorff}
	Let $(K,\hnear{k})$ be a relator space, and $x,y \in K$ be any two points in it. If there exist sets $X,Y \in K$ containing $x$ and $y$ respectivly, such that $\hnear{2}(X,Y)=1$, then $K$ is a Hausdorff($T_2$) space.
\end{lemma}
\begin{proof}
	A Hausdorff($T_2$) space is such that any two points ($x,y$) are contained in subsets($X,Y$) with a non-empty intersection. From \cite[axiom \textbf{(hP3)}]{ahmad2018maximal} two non-empty sets are Lodato hyperconnected if they have a non-empty intersection. Thus if two non-empty sets have an empty intersection, they are far and vice versa i.e. $\hnear{2}(X,Y)=1 \Leftrightarrow X \cap Y = \emptyset$. Hence proven. 
\end{proof}
The Hausdorff space as defined in Lemma~\ref{lem:hausdorff} can then be used as a basis for CW complex.
\begin{theorem}
	Let $(K,\hnear{k},\hsn{k})$ be a relator space satisfying the Hausdorff property presented in Lemma~\ref{lem:hausdorff}. If $K$ has a decomposition and satisfies:
	\begin{compactenum}[$1^o$]
		\item \textbf{Closure finiteness:} $\hsn{j+1}(cl A,\sigma^m_1,\cdots,\sigma^m_j)=0$, where $A= \sigma^n \in K$, $\sigma^m_i \in K$, $m,j$ are arbitrary finite nonnegative integers.
		\item \textbf{Weak topology:} $A \subset K$ is closed, provided $A \cap b$ is closed for all $b \in B$, where $B=\{cl \sigma^n: \sigma^n \in K\,and \, \hsn{2}(A,cl \sigma^n)=0\}$.  
		\end{compactenum}
	then $K$ is a CW complex.
\end{theorem}
\begin{proof}
	We require the two different relations, as Lodato $\hnear{k}$ is required to define the closure and strong hyperconnectedness $\hsn{k}$ to express the intersection properties of the CW complex. We will require \cite[axiom~\textbf{(snhN3)}]{ahmad2018maximal}, which states
	\begin{align*}
	\hsn{k}(A_1,\cdots,A_k)=0 \Rightarrow \mathop \bigcap \limits_{i=1,\cdots,k} A_i \neq \emptyset,
	\end{align*}  Let us proceed to prove the equivalence to the conditions of the CW complex.
	\begin{compactenum}[$1^o$]
		\item It can be seen that $A$	and $\sigma_i^m$ are both simplices(cells) in the space $K$. We can establish from \cite[axiom~\textbf{(snhN3)}]{ahmad2018maximal}
		that $\mathop \bigcap (clA,\sigma_1^m,\cdots,\sigma_j^m) \neq \emptyset$. Moreover, as $i$ is a finite nonnegative integer, it is either $0$ or a finite positive number. Thus, this statement states that the $clA$ can have a nonempty intersection with only finitly many cells other than $A$. Which is eqivalent to \textbf{Closure Finiteness} as stated in Sec.~\ref{sec:prelim}.
		
		\item From \cite[axiom~\textbf{(snhN3)}]{ahmad2018maximal} we can establish that $\hsn{2}(A,cl \sigma^n) \Rightarrow A \cap \sigma^n$. Thus $B$ is the set of closures of all the $\sigma^n \in K$ that have a nonempty intersection with $A$. Thus the condition states that $A$ is closed if all its nonempty intersections with other cells in $K$ are closed. This is equivalent to \textbf{Weak topology} as started in Sec.~\ref{sec:prelim}. 
	\end{compactenum}
Thus, we have established the equivalence of conditions as stated in this theorem with those listed in Sec.~\ref{sec:prelim}. Hence proved.
\end{proof}
\subsection{Hyperconnected genralization of path}\label{subsec:hyppath}
~\\
In the discussion that follows we attempt to formulate the notion of a path in terms of hyperconnectedness. Similar to proximity \cite{peters2016strongly} these relations can be used to define continuity. 
\begin{definition}\label{def:hypercnt}
	Let $(X,\hsn{k}),(Y,\hsn{k})$ be two relator spaces, $A_1,\cdots,A_n \in 2^X$ and  a function $f:X \rightarrow Y$. Then if,
	\begin{align*}
	\hsn{n}(A_1,\cdots,A_n)=0 \Rightarrow \hsn{n}(f(A_1),\cdots,f(A_n))=0
	\end{align*}
	the function $f$ is continuous in $\hsn{n}$ sense or $\hsn{n}-$continuous.
\end{definition}
Another associated notion with strong proximal continuity is that of strong proximal equivalence as defined in \cite[def.~3.1]{peters2016strongly}.
\begin{definition}\label{def:hypereqv}
	Let $(X,\hsn{k}),(Y,\hsn{k})$ be two relator spaces, $A_1,\cdots,A_n \in 2^X$ and  a function $f:X \rightarrow Y$. Then if,
	\begin{align*}
	\hsn{n}(A_1,\cdots,A_n)=0 \Leftrightarrow \hsn{n}(f(A_1),\cdots,f(A_n))=0
	\end{align*}
	the function $f$ is an equivalence in $\hsn{n}$ sense or $\hsn{n}-$equivalence.
\end{definition}
Here, we state a result regarding $\hsn{k}-$equivalence being a stronger condition than $\hsn{k}-$continuity.
\begin{theorem}
	Let $(X,\hsn{k}),(Y,\hsn{k})$ be two relator spaces and  a function $f:X \rightarrow Y$. Then, 
	\begin{align*}
		f \text{ is } \hsn{k}-\text{equivalence} \Rightarrow f \text{ is } \hsn{k}-\text{continuous}.
	\end{align*}
\end{theorem}
\begin{proof}
	Def.~\ref{def:hypereqv} implies that if $f$ is an $\hsn{k}-$equivalence, then $\hsn{k}(A_1,\cdots,A_n)=0 \Rightarrow \hsn{k}(f(A_1),\cdots,f(A_n))=0$. This, is the definition of $\hsn{k}-$continuity as per Def.~\ref{def:hypercnt}. Hence, proved. 
\end{proof}
\begin{remark}
We must point out an important distinction between $\hsn{k}-$continuous function and $\hsn{k}-$equivalnce. It is to be noted that by definition $\hsn{k}-$continuity of a map preserves only the nearness of sets i.e. it ensures the hyperconnected sets map to hyperconected sets. It does not ensure that two non-hyperconnected sets remain so under the map. $\hsn{k}-$equivalence is a stronger relation and it also ensures that the non-hyperconnected sets remain so under the map. Thus, $\hsn{k}-$equivalence has the same role in proximity spaces as the homeomorphism has in topological spaces. $\qquad \blacksquare$
\end{remark}

\begin{figure}
	\centering
	\includegraphics[scale=0.5,bb=120 450 320 670,clip]{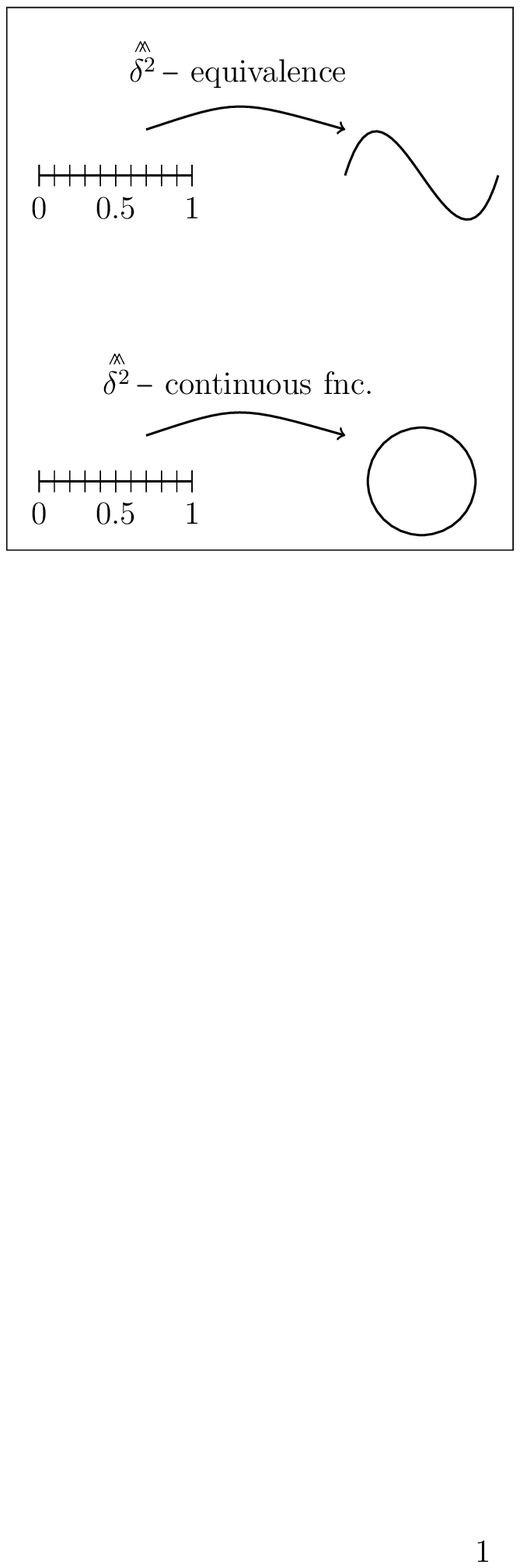}
	\caption{This figure illustrates that a $\hsn{2}-$ continuity preserves only the nearness of sets i.e. $\hsn{2}(A,B)=0 \Rightarrow \hsn{2}(f(A),f(B)=0$. This means that it allows gluing i.e. joining sets that are far, as shown by gluing the ends of a line interval to make a circle. Moreover, it illustrates that $\hsn{2}-$equivalence also preserves the fact that two sets are far. This is done by imposing the additional condition that $\hsn{2}(f(A),f(B)=0 \Rightarrow \hsn{2}(A,B)=0 $ This only allows for deformations that neither glue nor tear the space. }
	\label{fig:example4}
\end{figure}

We demonstrate the remark above by an example.
\begin{example}\label{ex:exmploop}
	Let us consider a map, $\pi$, from the line segment $[0,1]$ to a circle. For the purpose of this example we look at $\hsn{2}$ specifically. Consider the decomposition of $[0,1]$ as $\{A_i^{i+0.1}=[i-\eta,i+0.1-\eta]\}$ where $i=0,0.1,\cdots,0.9$ and $\eta$ is an arbitrarily small positive real number. It can be seen that $|i-j|=1 \Rightarrow \hsn{2}(A_i^{i+0.1},A_j^{j+0.1})=0$ and $|i-j|>1 \Rightarrow \hsn{2}(A_i^{i+0.1},A_j^{j+0.1})=1$. This is because the interiors of the adjacent sets intersect. It is obvious that the sets $\hsn{2}(A_0^{0.1},A_{0.9}^1)=1$, but the map $\pi$ glues $0$ to $1$. Thus, $\hsn{2}(\pi(A_0^{0.1}),\pi(A_{0.9}^1))=0$. The continuity of the line segment is not altered in anyother way except this gluing. Hence, the map $\pi$ in this case is $\hsn{2}-$coninuous as per Def.~\ref{def:hypercnt} but not $\hsn{2}-$equivalence as per Def.~\ref{def:hypereqv}. So, we can see that the notion of $\hsn{2}-$continuity being weaker than that of $\hsn{2}-$equivalence permits gluing. We can see this reflected in Fig.~\ref{fig:example4}.  $\qquad \blacksquare$  
\end{example}
Now let us formulate the idea of path connectedness in terms of hyperconnectedness relations. A \emph{path} between two points in a topological space, $x,y \in X$, is a continuous function $f:[0,1] \rightarrow X$ with $f(0)=x$ and $f(1)=y$. 
We define the notion of a $\hsn{2}-$ \emph{hyperconnected chain} in a relator space $(X,\hsn{k})$.
\begin{definition}\label{def:hyperchain}
	Let there be a family of sets $\{A_i\}_{i \in \mathbb{Z}^+} \subset X$, in a relator space $(X,\hsn{k})$ such that
	\begin{align*}
	 & |i-j| \leq 1  \Rightarrow  \hsn{2}(A_i,A_j)=0 \\
	 & |i-j|>1  \Rightarrow  \hsn{2}(A_i,A_j)=1.
	\end{align*}
	Then, $\{A_i\}_{i \in \mathbb{Z}^+}$ is a $\hsn{2}-$hyperconnected chain.
	 
\end{definition}
 We formulate the following lemma.
\begin{lemma}\label{lem:hyperchainembd}
	Let $\{A_i\}_{i \in \mathbb{Z}^+} \subset K$ be a $\hsn{2}-$ hyperconnected chain in a realtor space $(A,\hsn{k})$ and let $f:X \rightarrow Y$ be an $\hsn{2}-$equivalence. Then, $\{f(A_i)\}_{i\in \mathbb{Z}^+}$ is a $\hsn{2}-$ hyperconnected chain in the relator space $(Y,\hsn{k})$.
\end{lemma}
\begin{proof}
	Since $\{A_i\}_{i \in \mathbb{Z}^+}$ is $\hsn{2}-$hyperconnected chain, it follows from def.~\ref{def:hyperchain} that $|i-j|\leq 1 \Rightarrow \hsn{2}(A_i,A_j)=0$. As $f:X \rightarrow Y$ is an $\hsn{2}-$equivalence, def.~\ref{def:hypereqv} implies that $\hsn{2}(A_i,A_j) \Rightarrow \hsn{2}(f(A_i),f(A_j))=0$. Thus, $|i-j| \leq 1 \Rightarrow \hsn{2}(f(A_i),f(A_j))=0$. The function $f$ being an $\hsn{2}-$equivalence also dictates that $\hsn{2}(f(A),f(B)) \Rightarrow \hsn{2}(A,B)$. This means that if the images of two sets under $f$ are hyperconnected, their domains must also be hyperconnected. In other words if domains are not hyperconnected the images are also not. Using this and $|i-j|\leq 1 \Rightarrow \hsn{2}(A_i,A_j)=0$ from Def.~\ref{def:hyperchain}, we can conclude that $|i-j|>1 \Rightarrow \hsn{2}(f(A_i),f(A_j))=1$. Thus, from def.~\ref{def:hyperchain} $\{f(A_i)\}_{i \in \mathbb{Z}^+}$ is a $\hsn{2}-$hyperconnected chain as both the conditions have been established. 
\end{proof}
Let us consider the decomposition of $[0,1]$ given by the sets in $2^{[0,1]}$. \emph{Path} is homeomorphic to $[0,1]$, and $[0,1]$ can be expressed as $\hnear{2}-$hyperconnected chain $\{A_i\}_{i=1,\cdots,n} \subset [0,1]$, starting at $A_0=0$ and ending at $A_n=1$. In a topological space the path is the embedding of $[0,1]$. In the context of a hyperconnected relator space we consider the embedding of the hyperconnected chain,$\{A_i\}_{i \in \mathbb{Z}^+}$ in the relator space.
\begin{definition}\label{def:hyperpath}
	Let $(X,\hsn{k})$ be a relator space and $x,y \in X$ be two points in it. Then, a $\hsn{2}-$equivalence function $f:[0,1] \rightarrow X$ such that $f(0)=x,f(1)=y$ is a path from $x$ to $y$. 
\end{definition} 
We formulate the following theorem, which demonstrate that a path is equivalent to the existance of a $\hsn{2}-$hyperconnected chain.
\begin{theorem}\label{thm:pathequivchain}
	Let $(X,\hsn{k})$ be a relator space and $x,y \in X$ be two points. Then a path from $x$ to $y$ is equivalent to the existence of a $\hsn{2}-$hyperconnected chain in $X$,$\{B_i\}_{i=1, \cdots, n}$ such that $x \in B_1$ and $y \in B_n$.
\end{theorem}
\begin{proof}
	As from Def.~\ref{def:hyperpath}, the path from $x$ to $y$ is a $\hsn{2}-$equivalence, $f:[0,1] \rightarrow X$, such that $f(0)=x$ and $f(1)=y$. We can decompose the interval $[0,1]$ as a family of intervals $\{A_i^{i+\frac{1}{n}}=[i-\eta,i+\frac{1}{n}-\eta]\}_i$ for $i=0,\frac{1}{n},\frac{2}{n},\cdots,\frac{n-1}{n}$ and $\eta \in \mathbb{R}^+$ is arbitrarily small. As, we can see that $|i-j|\leq 1 \Rightarrow \hsn{2}(A_i^{i+\frac{1}{n}},A_j^{j+\frac{1}{n}})=0$ and $|i-j|> 1 \Rightarrow \hsn{2}(A_i^{i+\frac{1}{n}},A_j^{j+\frac{1}{n}})=1$ for this particular decomposition. This is because the interiors of adjacent intervals intersect. From Def.~\ref{def:hyperchain} it can be established that $\{A_i^{i+\frac{1}{n}}=[i,i+\frac{1}{n}]\}_i$ is a $\hsn{2}-$hyperconnected chain in $[0,1]$. As the path $f$ is an $\hsn{2}-$equivalence, thus using the lemma \ref{lem:hyperchainembd} we can conclude that there is a $\hsn{2}-$hyperconnected chain in the relator space $(Y,\hsn{k})$. The chain can be represented as $\{B_i\}_{i=1,\cdots,n}$ where $B_i=f(A_{(i-1)/n}^{i/n})$. Moreover, as we have established earlier in this proof that $f(0)=x$ and $0 \in A_0^{1/n}$, thus $x \in B_1$. Similarly $f(1)=y$ and $1 \in A_{(n-1)/n}^1$, thus $y \in B_n$. Hence proved.  
\end{proof}
Thus, we have established the equivalence of a path between two points in a space to the existence of hyperconnected chain such that the first and last sets in it, contain each of the points.

Let us now consider relaxed version of \emph{path}, than the one defined in Def.~\ref{def:hyperpath}. We start by considering a more relaxed version of $\hsn{2}-$hyperconnected chain defined in Def.~\ref{def:hyperchain}. We define the notion of a hyperconnected link in a relator space.
\begin{definition}\label{def:hyperlink}
	Let there be a family of sets $\{A_i\}_{i \in \mathbb{Z}^+} \subset X$, in a relator space $(X,\hsn{k})$ such that
	\begin{align*}
	& |i-j| \leq 1  \Rightarrow  \hsn{2}(A_i,A_j)=0 
	\end{align*}
	Then, $\{A_i\}_{i \in \mathbb{Z}^+}$ is a $\hsn{2}-$hyperconnected link.	
\end{definition}

\begin{figure}
	\centering
	\includegraphics[scale=0.5,bb=120 450 500 670,clip]{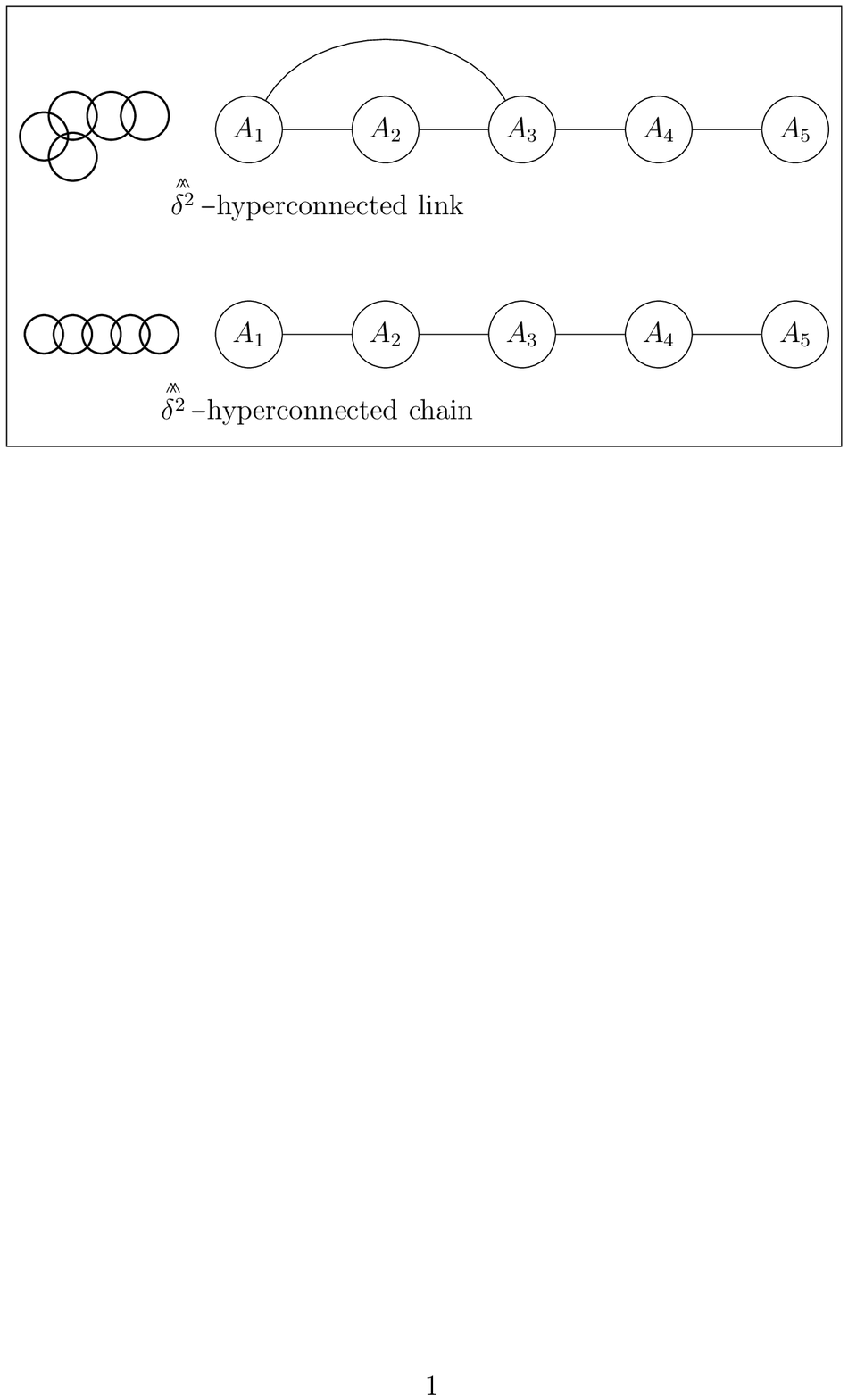}
	\caption{This figure illustrates a $\hsn{2}-$hyperconnected chains(Def.~\ref{def:hyperchain}) and $\hsn{2}-$hyperconnected link. The associated graphs are also shown. It can be seen that a hyperconnected chain is path graph $P_n$, while the hypperconnected link can have cycles.}
	\label{fig:hyplink}
\end{figure}

We have relaxed the condition that non-adjacent sets cannot be hyperconnected. We can easily illustrate this by using the notion of a graph to represent $\hsn{2}-$hyperconnected chains and links. The $\{A_i\}_{i \in \mathbb{Z}^+} \subset X$ fromt he Defs.~\ref{def:hyperchain} and \ref{def:hyperlink} are considered as nodes. We add an edge between the nodes if $\hsn{2}(A_i,A_j)=0$. We define a few terms from graph theory that can help us understand both these structures. A \emph{graph} consists of  vertices and edges between them. We consider only \emph{simple graphs}, that have no loops(edge from a vertex to itself) or multiple edges between the same node. A \emph{cycle} is a path, interms of edges, from a vertex back to itself passing through other vertices. In a cycle no edge is repeatable, and no vertex except the starting and ending vertices is repeatable. A graph with no cycles is a \emph{tree}. The number of edges connected to a particular vertex is its \emph{degree}.  
\begin{definition}\label{def:pathgraph}
	A path graph, $P_n$, is a tree with two nodes of degree $1$ and the remaining $n-2$ nodes of degree $2$. $P_n$ is a path graph with $n$ vertices.
\end{definition}
Here we present an important result.
\begin{theorem}\label{thm:pathgraph}
	Let $(X,\hsn{2})$ be a relator space and $\{A_i\}_{i =1,\cdots,n} \subset X$ be a family of subsets that forms a $\hsn{2}-$hyperconnected chain. If we construct a graph where $\{A_i\}_{i \in \mathbb{Z}^+}$ are the vertices and there is an edge for every $\hsn{2}(A_i,A_j)=0$, then this graph is a path graph, $P_n$. 
\end{theorem}
\begin{proof}
	As $\{A_i\}_{i=1,\cdots,n}$ form a $\hsn{2}-$hyperconnected chain, by Def.~\ref{def:hyperchain} there exist edges for each pair $(A_i,A_j)$, such that $|i-j| \leq 1$. We can see that for each $A_1$ and $A_n$ there is only one edge which connects to $A_2$ and $A_{n-1}$ respectively. This is as there are no $A_0$ or $A_{n+1}$. For all the other $A_i$, the condition for existence of path is satisfied twice as $\hsn{2}(A_{i-1},A_i)=0$ and $\hsn{2}(A_i,A_{i+1})=0$. Thus, there are two edges that connect to each of the $A_i$ such that $i \not \in \{1,n\}$, thus they have degree $2$. Moreover, by Def.~\ref{def:hyperchain} there exists no edge for each pair $(A_i,A_j)$ such that $|i-j|>1$. Thus, the only way to come back to a particular vertex is to come back to it via the edge used to leave it. This, violates the condition of a cycle in a simple graph which restricts the repitition of an edge. Thus, no cycles exist in this graph making it a tree. Hence, the graph is a tree with two nodes of degree $1$ and the remainig $n-2$ nodes of degree of $2$. According to Def.~\ref{def:pathgraph}, this graph is a $P_n$ or a path graph of $n$ vertices. 
\end{proof}  
 We now establish the analog of Lemma~\ref{lem:hyperchainembd} for $\hsn{2}-$hyperconnected links.
 \begin{lemma}\label{lem:hyperlinkembd}
 	Let $\{A_i\}_{i \in \mathbb{Z}^+} \subset K$ be a $\hsn{2}-$ hyperconnected link in a realtor space $(A,\hsn{k})$ and let $f:X \rightarrow Y$ be an $\hsn{2}-$continuous function. Then, $\{f(A_i)\}_{i\in \mathbb{Z}^+}$ is a $\hsn{2}-$ hyperconnected link in the relator space $(Y,\hsn{k})$.
 \end{lemma}
\begin{proof}
	Since $\{A_i\}_{i \in \mathbb{Z}^+}$ is $\hsn{2}-$hyperconnected link, it follows from def.~\ref{def:hyperlink} that $|i-j|\leq 1 \Rightarrow \hsn{2}(A_i,A_j)=0$. As $f:X \rightarrow Y$ is an $\hsn{2}-$continuous function, def.~\ref{def:hypercnt} implies that $\hsn{2}(A_i,A_j) \Rightarrow \hsn{2}(f(A_i),f(A_j))=0$. Thus, $|i-j| \leq 1 \Rightarrow \hsn{2}(f(A_i),f(A_j))=0$. Hence $\{A_i\}_{i \in \mathbb{Z}^+}$ is a $\hsn{2}-$hyperconnected link in the relator space $(Y,\hsn{2})$ as per the Def.~\ref{def:hyperlink}
\end{proof}
Now let us try to relax the definition of path given in Def.~\ref{def:hyperpath}. We start by considering the decomposition of $[0,1]$ into $n$ intervals as $\{A_i^{i+1}=[i-\eta,i+0.1-\eta]\}$ for $i=0,\cdots,n-1$ and arbitrarily small $\eta \in \mathbb{R}^+$. These, can easily be confirmed as a $\hsn{2}-$hyperconnected link. Because for $|i-j|=1$, the interors of intervals intersect i.e. $int(A_i)\cap int(A_j) \neq \emptyset$. This leads to $\hsn{2}(A_i,A_j)=0$. Thus, according to Def.~\ref{def:hyperlink} this decomposition is a hyperconnected link. Thus, if there is a $\hsn{2}-$continuous fucntion such that $f(0)=x$ and $f(1)=y$, we get a hyperconnected link $\{f(A_i^{i+1})\}$ for $i=0,\cdots,n-1$, as per Lem.~\ref{lem:hyperlinkembd}. Moreover, as $0 \in A_0^{0.1}$ and $1 \in A_{n-1}^n$, $f(0)=x \in f(A_0^{0.1})$, $f(1)=y \in f(A_{n-1}^n)$ repectivley. Thus, we have link between the points $x$ and $y$. As $\{A_i^{i+1}\}_i$ gives a path going from $0$ to $1$, where $\hsn{2}(A_i,A_j)=0$ allows moving from $A_i$ to $A_j$ due to $int(A_i) \cap int(A_j) \neq \emptyset$. The chain $\{f(A_i^{i+1})\}_i$ gives a path from $x$ to $y$ as according to Def.~\ref{def:hypercnt} the $\hsn{2}-$continuous fucntion $f$ preserves hyperconnectedness of sets. 

What is the difference between the path as defined by Def.~\ref{def:hyperpath} and this notion based on $\hsn{2}-$connected link and $\hsn{2}-$continuous function? The answer has been illustrated in example.~\ref{ex:exmploop} and Figs.~\ref{fig:example4} \& \ref{fig:hyplink}.
Let us explain this with the help of an example.
\begin{example}
	Fig.~\ref{fig:hyplink} illustrates a $\hsn{2}-$hyperconnected link and a $\hsn{2}-$hyperconnected chain. It can be seen that in both the cases we can go from $A_1$ to $A_5$. In case of the chain there is only one way $(A_1,A_2,A_3,A_4,A_5)$. In case of the link there are multiple ways,
	\begin{align*}
	& (A_1,A_2,A_3,A_4,A_5) \\
	& (A_1,A_3,A_4,A_5)\\
	& (A_1,A_2,A_3,A_1,A_2,A_3,A_4,A_5)
	\end{align*}
	Thus, in the case of a link we have mulitple paths, including the one that was yielded by the chain. This is due to fact that in the case of a link(Def.~\ref{def:hyperlink}), the connection between $A_i$ and $A_j$ for $|i-j|>1$ is not restricted. Because, the condition $|i-j|>1 \Rightarrow \hsn{2}(A_i,A_j)$ is not imposed in Def.~\ref{def:hyperlink}. 
\end{example}
Thus, we can provide a definition of a path that is more relaxed than the one presented in Def.~\ref{def:hyperpath}.
\begin{definition}\label{def:loop}
Let $(X,\hsn{k})$ be a relator space and $x,y \in X$ be two points in it. Then, a $\hsn{2}-$continuous function $f:[0,1] \rightarrow X$ such that $f(0)=x,f(1)=y$ is a path from $x$ to $y$ that allows loops and self interesections.	
\end{definition}
Here, we present a couple of definitions from graph theory that will help in understanding the subsequent result. A \emph{subgraph} is graph obtained from an other graph by using a subset of its vertices and edges. If the subgraph contains all the vertices of the parent graph, it is a \emph{spanning subgraph}.
\begin{theorem}\label{thm:subgraph}
	Let $(X,\hsn{k})$ be a relator space and $\{A_i\}_{i \in \mathbb{Z}^+}$ be a family of sets where $A_i \subset X$. Let $\mathcal{L}(\{A_i\}_i)$ be a $\hsn{2}-$hyperconnected link and $\mathcal{C}(\{A_i\}_i)$ be a $\hsn{2}-$hyperconnected chain consisting of $\{A_i\}_{i \in \mathbb{Z}^+}$. We construct a graph for each of the structures in which the $\{A_i\}_i$ are the nodes and there is a edge for each $\hsn{2}(A_i,A_j)=0$. $G(\mathcal{L})$ is the graph for $\mathcal{L}(\{A_i\}_i)$  and $G(\mathcal{C})$ is the graph for $\mathcal{C}(\{A_i\}_i)$. Then $G(\mathcal{C})$ is a spanning subgraph of $G(\mathcal{L})$.
\end{theorem}
\begin{proof}
	From the definition of $\hsn{2}-$hyperconnected chain, Def.~\ref{def:hyperchain}, we can see that $|i-j| \leq 1 \Rightarrow \hsn{2}(A_i,A_j)=0$ which is same as the condition for $\hsn{2}-$hyperconnected link,  Def.~\ref{def:hyperlink}. Thus, $G(\mathcal{L})$ has all the edges in $G(\mathcal{C})$. Furthermore, as Def.~\ref{def:hyperlink} does not impose the condition $|i-j|>1 \Rightarrow \hsn{2}(A_i,A_j)=1$, the graph $G(\mathcal{L})$ can have edges not in $G(\mathcal{C})$. Moreover, by definition the vertex set of $G(\mathcal{C})$ and $G(\mathcal{L})$ is the same i.e. $\{A_i\}_{i \in \mathbb{Z}^+}$. Thus, $G(\mathcal{C})$ is a spanning subgraph of $G(\mathcal{L})$. 
\end{proof}  
 We have defined the path(Def.~\ref{def:hyperpath}) as being equivalent to a $\hsn{2}-$hyperconnected chain and we have relaxed this notion to that of a $\hsn{2}-$hyperconnected link in Def.~\ref{def:hyperlink}. Thm.~\ref{thm:subgraph} states that the relaxed notion of a path contains the original path in it as a subgraph. 
 
 \section{Applications}
 In this section we consider the application of hyperconnected paths to define the notion of cycles connecting the centroids of triangles in a triangulated image. The triangulation is a CW complex that can be equipped with $\hsn{2}$ to result in a relator space as in \cite{ahmad2018maximal}. We start with a general notion of a cycle obtained by the joining centroids of selected triangles to illustrate the general methodology. This can then be specified to the case of maximal centroidal vortices \cite[Def.$4$, Thm.~$7$]{ahmad2018maximal}.
 
 Once we have selected the triangles that are to contribute their centroids as a vertices of the cycle, we have to construct the cycle. As defined previously, cycle is a path from the starting vertex back to itself traversing the other specified vertices in the process. To keep things simple we assume that the path we choose has no self intersection and loops. Thus, we refer back to Def.~\ref{def:hyperpath}, that specifies the path as an embedding of $[0,1]$ into the space under a $\hsn{2}-$equivalence. We have established in Thm.~\ref{thm:pathequivchain}, that the existence of a path between two points in a space is equivalent to the existence of a hyperconnected chain $\{A_i\}_{i=1,\cdots,n}$ such that $A_1$ contains the starting and $A_n$ the terminating vertex. We formalize this process as a sewing operator, which was introduced as a building block of physical geometry in \cite{Peters2016arXivProximalPhysicalGeometry}.
 \begin{definition}\label{def:sew}
 	Let $(X,\hsn{2})$ be a relator space, $x,y$ be two points in $X$ and $n \in \mathbb{Z}^+$. We start by choosing two sets $A_1, A_n \in 2^X$ such that $x \in A_1, y \in A_n$ while restricting $x \not \in A_n, y \not \in A_1$. We construct the following sets,
 	\begin{align*}
 	\text{for } i=&2,\cdots,n-2 \text{:}\\
 	\mathcal{A}_i=&\{S \in 2^X: (\forall T \in \mathcal{A}_{i-1}) \Rightarrow (int(S) \cap int(T) \neq \emptyset), \, (\forall U \in \mathcal{A}_{i-j}, j>1) \Rightarrow \\
    &(int(S) \cap int(U) \neq \emptyset) and \; (\forall V \in \mathcal{A}_n) \Rightarrow (int(S) \cap int(V)=\emptyset)\}\\
 	\text{for } i=&n-1 \text{:}\\
 	\mathcal{A}_i=&\{S \in 2^X: (\forall T \in \mathcal{A}_{i-1}) \Rightarrow (int(S) \cap int(T) \neq \emptyset), \, (\forall U \in \mathcal{A}_{i-j}, j>1) \Rightarrow \\
 	&(int(S) \cap int(U) \neq \emptyset) and \; (\forall V \in \mathcal{A}_n) \Rightarrow (int(S) \cap int(V)\neq \emptyset)\}\\
 	\mathcal{A}_1=&A_1, \quad \mathcal{A}_n=A_n 
 	\end{align*}
 	where $int(X)$ is the interior of set $X$. Then,
 	\begin{align*}
 	\mathop{sw} \limits_n(x,y)=\{A_i:A_i \in \mathcal{A}_i\}_{i=1,\cdots,n}
 	\end{align*}
 	is the sewing operator of degree $n$ between points $x,y \in X$ is a family of sets.  
 \end{definition}
Here, we present an other result that directly follows.
\begin{theorem}
	Let $(X,\hsn{2})$ be a relator space and $x,y$ be two points in $X$. Then, $\mathop{sw} \limits_n(x,y)$ is a $\hsn{2}-$hyperconnected chain.
\end{theorem}
\begin{proof}
	By Def.~\ref{def:sew} it can be seen that $\mathop{sw} \limits_n(x,y)$ yields a family of sets. By definition, $int(A_i) \cap int(A_{i-1}) \neq \emptyset$ and $int(A_i) \cap int(A_{i+1}) \neq \emptyset$ for $i=2,\cdots,n-1$. Moreover by definition, $int(A_1) \cap int(A_2) \neq \emptyset$ and $int(A_n) \cap int(A_{n-1}) \neq \emptyset$. Thus, $|i-j|\leq 1 \Rightarrow \hsn{2}(A_i,A_j)=0 $ as per \cite[axiom~\textbf{(snhN3)}]{ahmad2018maximal}. Moreover, by definition, it is also clear that $|i-j|>1 \Rightarrow int(A_i) \cap int(A_j) = \emptyset$. Thus, by \cite[axiom~\textbf{(snhN3)}]{ahmad2018maximal} we have $|i-j|>1 \Rightarrow \hsn{2}(A_i,A_j)=1$. From Def.~\ref{def:hyperchain}, we can see that these two conditions are required for a family of sets to be a $\hsn{2}-$hyperconnected chain. Hence, proved that $\mathop{sw} \limits_n(x,y)$ is a hyperconnected chain.    
\end{proof}
From, this we can now define a cycle which consists of vertices, $(v_1,v_2,\cdots,v_n)$. Then, 
\begin{definition}\label{def:cycimg}
	Let $(X,\hsn{2})$ is a relator space and $v_i \in X$ for $i \in \mathbb{Z}^+$. Then,
	\begin{align*}
	cyc(\{v_i\}_{i \in \mathbb{Z}^+})=\{\mathop{sw} \limits_n(v_n,v_1),\mathop{sw} \limits_n(v_i,v_{i+1})\,for i=1,\cdots,n-1\}
	\end{align*}
	is the cycle starting and ending at $v_1$ and traversing the remaining $v_i$ for $i=2,\cdots,n$.
\end{definition}
\begin{figure}
	\centering
	\begin{subfigure}[Original Figure]
		{\includegraphics[width=2in]{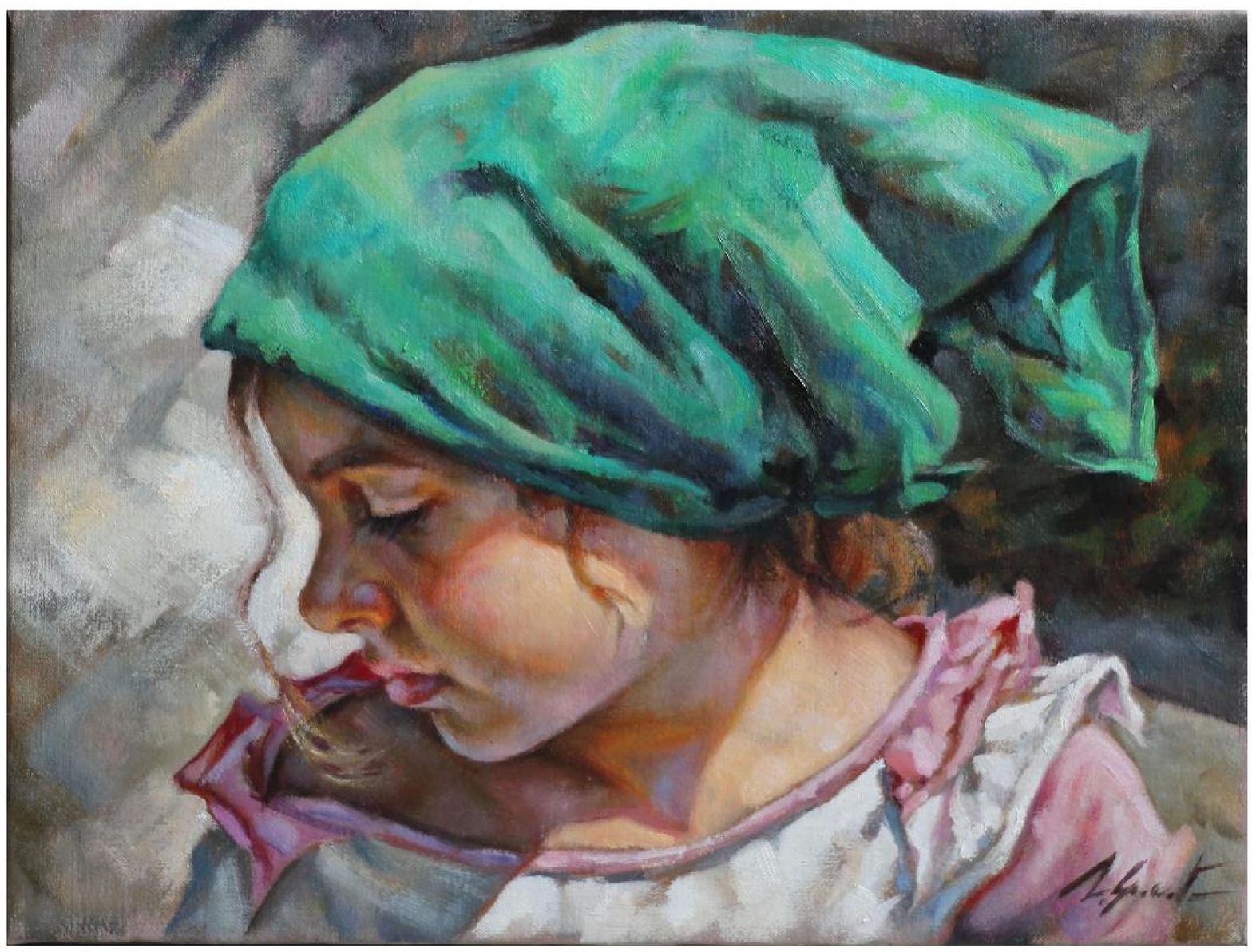}
			\label{subfig:orig1}}
	\end{subfigure}
	\begin{subfigure}[Cycle in a triangulated image]
		{\includegraphics[width=2in]{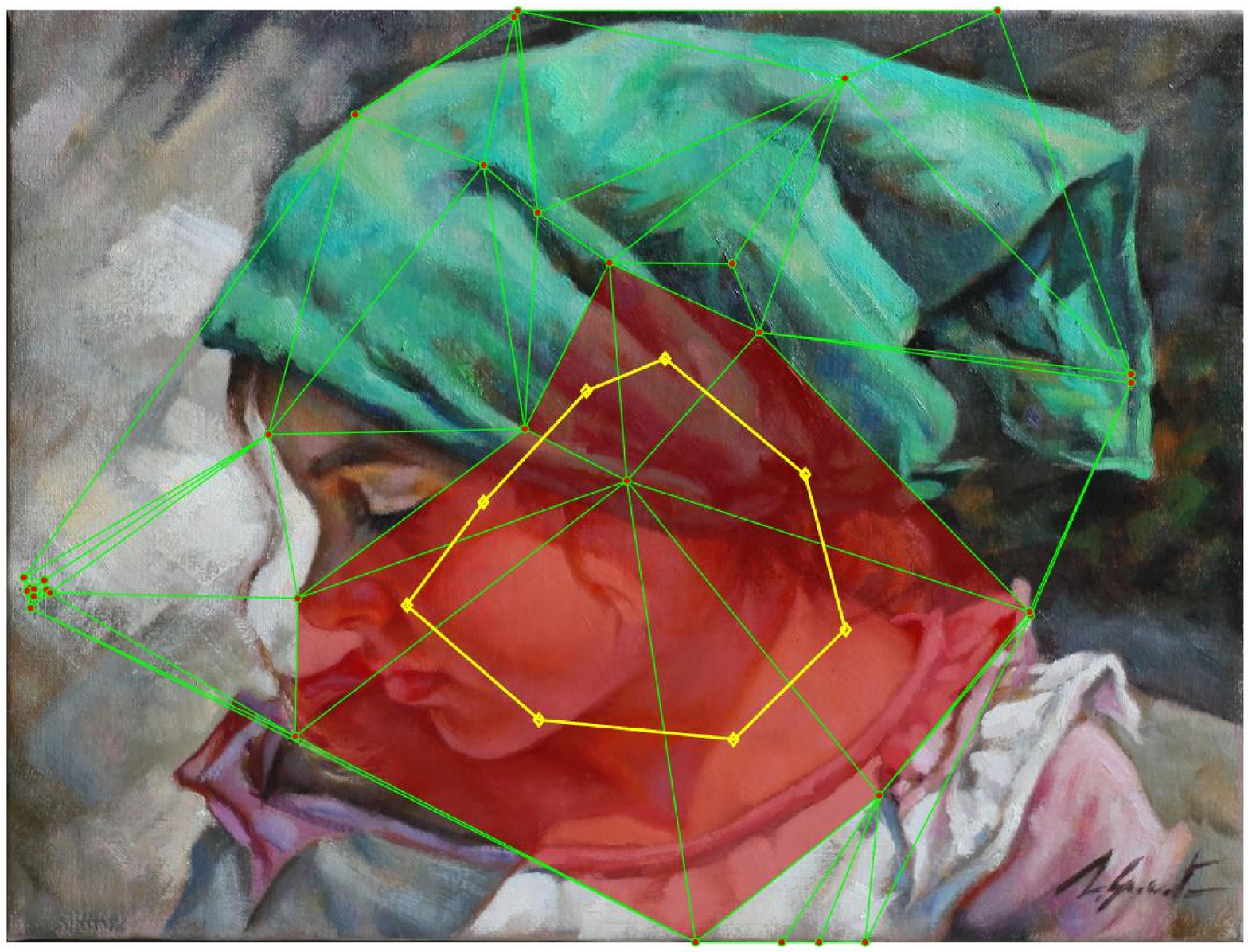}
			\label{subfig:cyc1}}
	\end{subfigure}
	\begin{subfigure}[Original Figure]
		{\includegraphics[width=2in]{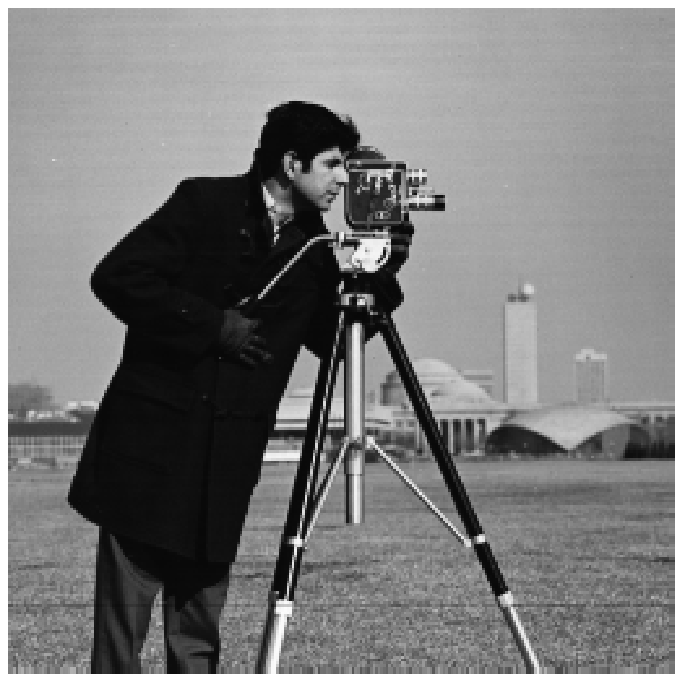}
			\label{subfig:orig2}}
	\end{subfigure}
	\begin{subfigure}[Cycle in a triangulated image]
		{\includegraphics[width=2in]{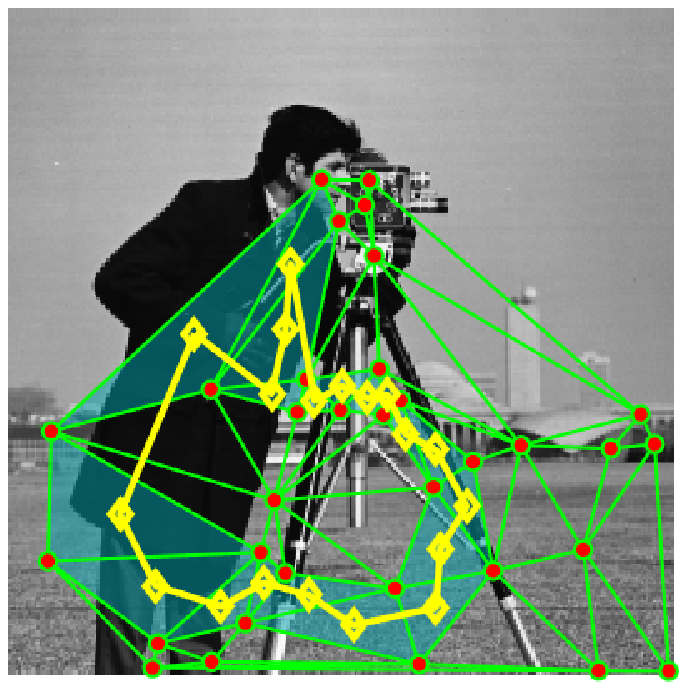}
			\label{subfig:cyc2}}
	\end{subfigure}
	\caption{Fig.~\ref{subfig:orig1} represents the original image and Fig.~\ref{subfig:cyc1} illustrates a cycle drawn an image using centroids of triangulation. Fig.~\ref{subfig:orig2} is the \emph{cameraman.tiff} image from MATLAB and Fig.~\ref{subfig:cyc2} illustrates a cycle drawn using centroids of triangles in the triangulation of the image.}
	\label{fig:cycapp}	
\end{figure}

Thus, we have defined the notion of a cycle in the triangulation as a $\hsn{2}-$hyperconnected chain, via the notion of a sewing operator. It can be seen that the cycle as defined in Def.~\ref{def:cycimg} is a concatenation of hyperconnected chains or as from the equivalence established in Thm.~\ref{thm:pathequivchain}, as a concatenation of paths without self intersections. Here, a point to note is that a cycle is homeomorphic to a circle. We have seen in Fig.~\ref{fig:example4} the interval $[0,1]$ can be mapped onto a circle, but, as mentioned in example~\ref{ex:exmploop}, this map involves a gluing and is thus a $\hsn{2}-$continuous map. Here, the cycle is built as a concatenation of embeddings of $[0,1]$ under $\hsn{2}-$equivalences. Even though the individual sections of the cycle between two vertices is a $\hsn{2}-$hyperconnected chain, the concatenation is not a $\hsn{2}-$hyperconnected chain, since the last set is hyperconnected to the first set. Hence, the cycle is in itself a $\hsn{2}-$hyperconnected link.

Let us now look at the how these cycles are embedded in the triangulation of a digital image. The triangluation is generated by selecting keypoints from the image. We define the concept of a \emph{hole}, which is a region containing vertexes wiht the same description (such vertexes, in this case, are pixels with uniform intensity).   The centroids of these holes are considered as the seed points of triangulation. Here we draw cycles on the images using centroids of triangles in the spoke complexes, as defined in \cite[Def.$4$]{ahmad2018maximal}.

We present two images and cycles in them in Fig.~\ref{fig:cycapp}.   Fig.~\ref{subfig:orig1} is a portrait depicting a girl\footnote{Many thanks to Alesssandro Granata, Salerno, Italy for his permission to use his painting in this study of image geometry.}. From the triangulation we select the triangles that are in a maximal nuclear cluster, and draw on the triangulated image in Fig.~\ref{subfig:cyc1}. A collection of triangles that share a non empty intersection is called an \emph{Alexandroff nerve}, and the nerve with largest number of triangles is the \emph{maximal nuclear cluster}. The common intersection of the triangles in maximal nuclear cluster is called the \emph{nucleus}.  Next, consider the cameraman image \emph{cameraman.tiff} in the stock images of MATLAB depicted in Fig.~\ref{subfig:orig2}. For this image we construct a triangulation similar to the previous image. In this case, to construct the cycle, we use the triangles in $skcx_1$ (spoke complex of degree $1$), which are the triangles sharing a nonempty intersection with the triangles in maximal nuclear cluster but have an empty intersection with the nucleus itself. This cycle is depicted in Fig.~\ref{subfig:cyc2} and lies on the body and tripod of the cameraman in the image.

\section{Conclusion}
In this paper we revist the concepts of CW complex and paths in a topological space, in view of hyperconnectedness, that is a genaralization of proximity relations. The main results of the paper include a hyper-connectedness form of CW complex, and the existence of paths with or without self intersections. We equate these notions to the existence of hyperconnected chains and links in a relator space $(X,\hsn{2})$. We conclude with the application of these concepts to define cycles in a trinagulated image.
   
\bibliographystyle{amsplain}
\bibliography{HCVrefsnew}
\end{document}